\documentclass[reqno, oneside]{amsart}

\usepackage{mymacro}
\usepackage{bookmark}
\usepackage{fancyhdr}

\usepackage{hyperref}
\usepackage{url}
\hypersetup{breaklinks=true}

\usepackage{comment}
\makeatletter
\@namedef{subjclassname@2020}{%
  \textup{2020} Mathematics Subject Classification}
\makeatother

\author[Gilabert]{Mart\'in Gilabert Vio}
\address{\hskip-\parindent{}
	Institut Camille Jordan, Université Claude Bernard Lyon 1, 43 boulevard du 11 novembre
1918, 69622 Villeurbanne, France.}
\email{gilabert@math.univ-lyon1.fr // martingilabertvio@gmail.com}
\date{\today}


\title{Dynamical Tits alternative for groups of almost automorphisms of trees}

\begin{document}
	
\begin{abstract}
We prove a dynamical variant of the Tits alternative for the group of almost automorphisms of a locally finite tree $\cT$: a group of almost automorphisms of $\cT$ either contains a nonabelian free group playing ping-pong on the boundary $\partial \cT$, or the action of the group on $\partial \cT$ preserves a probability measure. This generalises to all groups of tree almost automorphisms a result of S. Hurtado and E. Militon for Thompson's group $V$, with a hopefully simpler proof.
\end{abstract}
	
\maketitle
\section{Context and contributions}
The Tits alternative is a celebrated theorem by J. Tits \cite{titsFree} that shows a sharp dichotomy for linear groups over a field of characteristic zero: either they are virtually solvable or they contain a nonabelian free group. A group $G$ is said to satisfy the Tits alternative if for every subgroup $H$ of $G$, $H$ is virtually solvable or contains a nonabelian free group. This group property has been established for a great deal of countable groups (see the references in \cite[Section II B, Complement 42]{deLaHarpeTopics}), usually by applying the Klein ping-pong lemma to exhibit free subgroups.

There are also many countable groups known to fail this alternative, as do many groups of homeomorphisms of compact spaces. For instance, the group $\mathrm{Homeo}(S^1)$ of homeomorphisms of the circle and the group of automorphisms $\mathrm{Aut}(\cT)$ of a regular tree $\cT$ of degree $\geq 3$: the former contains Thompson's group $F$ of piecewise affine dyadic homeomorphisms of $[0,1]$, the latter contains the first Grigorchuk group, and these subgroups are not virtually solvable and do not contain free groups (see \cite{CFP} and \cite{grigorchuk}, respectively). Nevertheless, these two examples satisfy a dynamical variant of this condition which we formulate as follows.

\begin{defn} \label{defn:dynamicalTits}
Let $X$ be a compact topological space and $G$ a group of homeomorphisms of $X$. We say that the action of $G$ on $X$ \emph{satisfies the dynamical Tits alternative} if for every subgroup $H$ of $G$ one of the following holds.
\begin{itemize}
	\item The action of $H$ preserves a regular probability measure on $X$.
	\item There exists a \emph{ping-pong pair} for the action of $H$, that is, there exist $g,h \in H$ and $U_1, U_2, V_1, V_2 \subset X$ disjoint open sets such that
	\begin{equation} \label{eq:pingpong}
		g(X\setminus U_1) \subseteq V_1\quad  \text{ and } \quad h(X \setminus U_2) \subseteq V_2.
	\end{equation}
\end{itemize}
\end{defn}

This dynamical alternative is a property of a group action, not merely of a group. Nonetheless, if $g,h \in H$ belong to a ping-pong pair, the ping-pong lemma shows that $g,h$ generate a nonabelian free group. Moreover, the conditions in Definition \ref{defn:dynamicalTits} exclude each other 
and it suffices to verify them on finitely generated $H \leq G$, see the beginning of the proof of Theorem \ref{teo:neretin}.

\begin{rmk}
Previous work \cite{malicetMiliton, hurtadoMiliton} involving this notion define the dynamical Tits alternative as a weaker condition, where every subgroup $H$ is required to preserve a probability measure or to contain a nonabelian free group. We prefer our definition since this weaker notion is not a dichotomy, and moreover all known proofs of the alternative yield the stronger condition. For instance, whenever $G \curvearrowright X$ satisfies Definition \ref{defn:dynamicalTits}, a subgroup $H \leq G$ preserves a probability measure on $X$ if and only if every pair of elements of $H$ preserve a common probability measure on $X$.
\end{rmk}

\begin{ejms}
A first family of examples comes from one-dimensional dynamics: the action on $S^1$ of the group of homeomorphisms of $S^1$ satisfies the dynamical Tits alternative by a theorem of G. Margulis \cite{margulisCircle}. A related example is the Higman-Thompson group $V$ acting on the triadic Cantor set, which also satisfies the alternative by work of S. Hurtado and E. Militon \cite[Theorem 1.3]{hurtadoMiliton}. A generalization of both statements is given in \cite[Theorem 1.3]{malicetMiliton},  where it is shown that for any compact $K \subset \R$, the defining action of the group of locally monotone homeomorphisms of $K$ satisfies the alternative. It is notable that the proof in \cite{malicetMiliton} finds sufficiently proximal elements on a group $G$ that does not preserve a measure on $K$ by running a random walk on $G$, whereas the arguments in \cite{margulisCircle, hurtadoMiliton} are ``deterministic''. Groups acting by homeomorphisms on dendrites also satisfy the alternative by work of B. Duchesne and N. Monod \cite[Theorem 1.6]{DM}.

A second family of examples consists of groups acting on the boundary of Gromov-hyperbolic spaces: a first elementary instance of this class is the action of automorphism group of a locally finite tree $\cT$ on its boundary $\partial \cT$, as follows easily from the well-known dynamical classification of subgroups of $\mathrm{Aut}(\cT)$, see \cite{titsArbre}. More generally, if $(M,d)$ is a Gromov-hyperbolic and proper metric space such that $\mathrm{Isom}(M,d)$ acts cocompactly on $M$, then the action of $\mathrm{Isom}(M,d)$ on the Gromov boundary $\partial M$ satisfies the dynamical Tits alternative (see \cite{CCMT}, and also \cite[Theorem 1.10]{aounSert} for a probabilistic version).
\end{ejms}

This note is concerned with almost automorphism groups of locally finite trees, which is a large family of locally compact and totally disconnected groups that arise as follows: let $\cT$ be a locally finite rooted tree and $\partial \cT$ be its space of ends. The group of rooted tree automorphisms $\mathrm{Aut}_\mathrm{r}(\cT)$ acts on $\partial \cT$ preserving the so-called visual metric. The group $\mathrm{AAut}(\cT)$ of almost automorphisms of $\partial \cT$ consists of all homeomorphisms of $\partial \cT$ that are local homothecies for this metric, that is, that locally rescale it. The natural group topology on $\mathrm{AAut}(\cT)$ is not the compact-open topology, but the unique group topology such that $\mathrm{Aut}(\cT)$ is a compact open subgroup of $\mathrm{AAut}(\cT)$. See Section \ref{sec:preliminaries} for more precise statements. 

If $d \geq 2$, $k \geq 1$ and $\cT_{d,k}$ is the rooted tree where the root has $k$ children and all other vertices have $d$ children, then $\mathrm{AAut}(\cT_{d,k})$ is known as a Neretin group. These groups are known to be simple \cite{kapoudjian}, compactly presented \cite{leBoudecCompact} and to contain the Higman-Thompson group $V_{d,k}$ as a dense subgroup. They are the first examples of compactly generated simple groups without lattices \cite{BCGM} and moreover admit no invariant random subgroups by a result of T. Zheng \cite{zheng}.

We show that the action of $\mathrm{AAut}(\cT)$ on $\partial \cT$ satisfies the dynamical Tits alternative for any locally finite rooted tree $\cT$.

\begin{thmA} \label{teo:neretin}
Let $\cT$ be a locally finite rooted tree. The action of $\mathrm{AAut}(\cT)$ on the boundary $\partial \cT$ satisfies the dynamical Tits alternative.
\end{thmA}

Some interesting groups to which Theorem \ref{teo:neretin} applies are the groups $V_G$ considered by V. Nekrashevych in \cite{nekrashevych}, where $G \leq \mathrm{Aut}_\mathrm{r}(\cT_{2,2})$ is a self-similar group. Here $V_G$ is the subgroup of $\mathrm{AAut}(\cT_{2,2})$ generated by $G$ and Higman-Thompson's $V$.

Fix a linear order on $\partial \cT$ that is compatible with the tree structure. We denote by $V_{\cT}$ the group of elements of $\mathrm{AAut}(\cT)$ acting in a locally order-preserving manner. For instance, $V_{\cT_{d,k}}$ is the Higman-Thompson group associated to $\cT_{d,k}$, and all topological full groups of irreducible infinite one-sided shifts of finite type are naturally subgroups of some $V_\cT$, see \cite[Theorem 3.8]{lederleV}. We call $V_\cT$ the \emph{Higman-Thompson group associated to $\cT$}, although this name is not standard. The proof of Theorem \ref{teo:V} below gives a shorter and hopefully more conceptual approach to \cite[Theorem 1.3]{hurtadoMiliton} when specialized to the Higman-Thompson group $V$.

\begin{thmA} \label{teo:V}
Let $H$ be a finitely generated subgroup of $V_\cT$. Then the action of $H$ on $\partial \cT$ has a finite orbit or admits a ping-pong pair.
\end{thmA}

We emphasize that Theorem \ref{teo:V} is not new, since the proof of \cite[Theorem 1.3]{hurtadoMiliton} shows that $V$ verifies the (slightly stronger) conclusion of Theorem \ref{teo:V} and general arguments allow to extend the result from $V$ to any group $V_\cT$.

Two important ingredients in both proofs are a characterization of relatively compact subgroups of $\mathrm{AAut}(\cT)$ by A. Le Boudec and P. Wesolek \cite{leBoudecWesolek} and a description of the dynamics of individual elements of $\mathrm{AAut}(\cT)$ by G. Goffer and W. Lederle \cite{GL} (building on work of O. Salazar-Díaz \cite{salazarDiaz}).

\subsection*{Acknowledgements}
The author wishes to thank Gil Goffer, Adrien Le Boudec, Waltraud Lederle, Nicolás Matte Bon, Mikael de la Salle and the anonymous referee for useful comments on a previous draft of this preprint. The author also thanks his advisor Nicolás Matte Bon for advice and encouragement, and Alejandra Garrido for pointing out \cite{GL}.

\section{Preliminaries} \label{sec:preliminaries}
We give some background on $\mathrm{AAut}(\cT)$, describe the dynamics of individual elements of $\mathrm{AAut}(\cT)$, recall the definition of the Vietoris topology on closed subsets of a space and fix some notation. For more details on this material, see \cite{GL, leBoudecWesolek, garncarekLazarovich}.

\subsection*{Notation}
Given a metric space $(X,d)$, $A \subseteq X$ and $\epsilon > 0$, we denote $A^\epsilon = \{x \in X \mid d(x,A)\leq \epsilon\}.$ When $\cT$ is a rooted tree, we write $\mathrm{Aut}_\mathrm{r}(\cT)$ for the group of tree automorphisms of $\cT$ that fix the root.

\subsection*{Almost automorphism groups of trees}
Let $\cT$ be a locally finite rooted tree with no leaves. We denote its root by $v_0$, and assume that all edges are directed away from $v_0$. A \emph{caret} is a subtree of $\cT$ consisting of a vertex, its children, and the edges between them. A subtree is \emph{complete} if it is a union of carets, and when $\cT_1, \cT_2$ are rooted complete subtrees of $\cT$ we denote by $\cT_2 \setminus \cT_1$ the union of all carets in $\cT_2$ that are not included in $\cT_1$. The set of leaves of a tree $\cT$ is denoted by $\cL \cT$.

Let $\partial \cT$ be the \emph{space of ends} of $\cT$, that is, the set of (one-sided) infinite directed paths starting at $v_0$. We equip $\partial \cT$ with the topology induced by the \emph{visual metric} defined as $d(\xi, \xi') = 2^{-N(\xi, \xi')}$ for $\xi, \xi' \in \partial \cT$ where $N(\xi, \xi') \in \N$ is the smallest integer $n$ such that $\xi_n \neq \xi'_n$. The space $(\partial \cT,d)$ is totally disconnected and compact, and its topology has a basis of clopen balls $\partial \cT_v = \{\xi \in \partial \cT \mid v \text{ is in } \xi\}$ where $v \in V(\cT)$.

An \emph{almost automorphism} of $\cT$ is a homeomorphism $g$ of $\partial \cT$ such that there exists a partition of $\partial \cT$ into clopen balls $D_1, \ldots, D_n$ and positive numbers $\lambda_1, \ldots, \lambda_n$ such that $d(g(y), g(z)) = \lambda_j d(y,z)$ for all $y,z \in D_j$. Such a partition is said to be \emph{admissible} for $g$. Another way of viewing an almost automorphism is the following: take $\cT_1, \cT_2$ finite subtrees of $\cT$ with root $v_0$, so $\cT \setminus \cT_1$, $\cT \setminus \cT_2$ are naturally rooted forests. Then any isomorphism $\overline{g} \colon \cT \setminus \cT_1 \to \cT \setminus \cT_2$ of rooted forests determines a $g \in \mathrm{AAut}(\cT)$, and conversely any almost automorphism arises in this manner, although not uniquely so.

Call an almost automorphism $g \in \mathrm{AAut}(\cT)$ \emph{elliptic} if there exists a partition $\cP$ of $\partial \cT$ into clopen balls that is admissible for $g$ and $g$-invariant, that is, such that $g\cP = \cP$. The following theorem is stated for Neretin groups in \cite{leBoudecWesolek}, but the proof for any locally finite tree $\cT$ is the same word by word.

\begin{prop}[{\cite[Corollary 3.6]{leBoudecWesolek}}] \label{prop:elliptic}
Let $H \leq \mathrm{AAut}(\cT)$ be finitely generated. The following are equivalent.
\begin{itemize}
	\item $H$ is relatively compact in $\mathrm{AAut}(\cT)$.
	\item Every element of $H$ is elliptic.
	\item There is a partition $\cP$ of $\partial \cT$ into clopen balls such that $\cP$ is admissible for every $h \in H$.
\end{itemize}
Moreover, if these conditions hold and $\cQ$ is any partition of $\cP$ into clopen balls such that $\cQ$ is admissible for every element of some generating set of $H$, then $\cP$ can be chosen to be finer than $\cQ$.
\end{prop}

Fix a family of total linear orders $\{<_v\}_{v \in V(\cT)}$ indexed by the vertices of $\cT$, where $<_v$ is an order on the children of $v \in V(\cT)$. This family induces a linear order $<$ on $\partial \cT$ by declaring $\xi < \xi'$ if $\xi, \xi' \in \partial \cT$ and $\xi_n <_{\xi_{n-1}} \xi'_n$, where $n = N(\xi, \xi') \in \N$. We call $g \in \mathrm{AAut}(\cT)$ a \emph{Higman-Thompson element} if there exists an admissible partition $\cP$ for $g$ such that $\restr{g}{D}$ is order-preserving for all $D \in \cP$. The subset of Higman-Thompson elements of $\mathrm{AAut}(\cT)$ is a group that we denote by $V_\cT$. We omit $\{<_v\}_{v \in V(\cT)}$ from the definition of $V_\cT$ to keep the notation uncluttered, but for non-regular trees $\cT$ the group $V_\cT$ may depend on the choice of the orders $\{<_v\}_{v \in V(\cT)}$.

Since the elliptic elements of $V_\cT$ are exactly the elements of finite order in $V_\cT$, as in \cite{leBoudecWesolek} Proposition \ref{prop:elliptic} yields the following result.

\begin{cor}[{\cite[Corollary 3.7]{leBoudecWesolek}}] \label{cor:torsion}
Any finitely generated subgroup of $V_\cT$ composed entirely of elliptic elements is finite.
\end{cor}

\subsection*{Dynamics of almost automorphisms}
We will make use of a description of the dynamics of individual elements of $\mathrm{AAut}(\cT)$, which is one of the subjects of \cite{GL}. Again, the proofs are given for $\cT = \cT_{d,k}$ but a careful reading shows that all arguments hold in the general case. 

 We need some definitions, following \cite{GL}: a \emph{tree pair} is a tuple $(\kappa, \cT_1, \cT_2)$ where $\cT_1, \cT_2 \subset \cT$ are finite complete trees with root $v_0$ and $\kappa \colon \cL \cT_1 \to \cL \cT_2$ is a bijection between the leaves of $\cT_1$ and the leaves of $\cT_2$. We say that a tree pair $(\kappa, \cT_1, \cT_2)$ is \emph{associated} to an element $g \in \mathrm{AAut}(\cT)$ if $g$ arises from an isomorphism of rooted forests $\overline{g} \colon \cT \setminus \cT_1 \to \cT \setminus \cT_2$ such that $\kappa = \restr{\overline{g}}{\cL \cT_1}$. In this case  $\{\partial \cT_v\}_{v \in \cL \cT_1}$ is an admissible partition for $g$. We will only consider tree pairs that are associated to some almost automorphism of $\cT$. All tree pairs are associated to almost automorphisms when $\cT = \cT_{d,k}$ but this is not true in general since the connected components of $\cT \setminus \cT_1$ and $\cT \setminus \cT_2$ need not be isomorphic.

Consider an orbit $\cO = \{u_0,\ldots, u_n\} \subseteq \cL \cT_1 \cup \cL \cT_2$ of the partial action of $\kappa$ on $\cL \cT_1 \cup \cL \cT_2$, that is, $\cO$ is such that $u_0,\ldots, u_{n-1} \in \cL \cT_1,\, u_1,\ldots, u_n \in \cL \cT_2$, $\kappa(u_j) = u_{j+1}$ for all $j = 0,\ldots, n-1$ and either
\begin{itemize}
	\item $u_0 \notin \cL \cT_2$ and $u_n \notin \cL \cT_1$, or
	\item $u_0 \in \cL \cT_2,\, u_n \in \cL \cT_1$ and $\kappa(u_n) = u_0$.
\end{itemize}
The orbit $\cO$ is said to be
\begin{itemize}
	\item an \emph{attracting chain} if $u_n$ is a descendant of $u_0$ in $\cT$, in which case $u_n$ is called the \emph{attractor} of the orbit,
	\item a \emph{repelling chain} if $u_0$ is a descendant of $u_n$ in $\cT$, in which case $u_0$ is called the \emph{repeller} of the orbit,
	\item a \emph{periodic chain} if $\kappa(u_n) = u_0$, and
	\item a \emph{wandering chain} if $u_0 \notin \cT_2$ (that is, $u_0$ is a descendant of some leaf in $\cL \cT_2$) and $u_n \notin \cT_1$ (that is, $u_n$ is a descendant of some leaf in $\cL \cT_1$).
\end{itemize}
These options are mutually exclusive. We say that $(\kappa, \cT_1, \cT_2)$ is a \emph{revealing tree pair} if each connected component of $\cT_1 \setminus \cT_2$ contains a repeller and each connected component of $\cT_2 \setminus \cT_1$ contains an attractor. In this case, if $C \subseteq \cT_1 \setminus \cT_2$ is a connected component containing a repeller $u_0$, the orbit $\{u_0,\ldots, u_n\}$ is such that $u_n$ is the root of $C$ and $u_0$ is the unique repeller in $C$. Similarly, if $C \subseteq \cT_2 \setminus \cT_1$ is a connected component containing an attractor $u_n$, the orbit $\{u_0, \ldots, u_n\}$ is such that $u_0$ is the root of $C$ and $u_n$ is the unique attractor in $C$. Revealing tree pairs were introduced by O. Salazar-Díaz in \cite{salazarDiaz} to describe the dynamics of individual elements of Thompson's $V$.

If $(\kappa, \cT_1, \cT_2)$ is a revealing tree pair, then every orbit $\cO$ is an attracting, repelling, periodic or wandering chain: assume that $\cO$ is not periodic, so it must end in an element $u_n \in \cL \cT_2 \setminus \cL \cT_1$ and begin in an element $u_0 \in \cL \cT_1 \setminus \cL \cT_2$. If $u_0 \in \cT_2$, then $u_0$ is the root of its component of $\cT_2 \setminus \cT_1$ and $u_n$ must be the attractor in this component. If $u_n \in \cT_1$, then $u_n$ is the root of its component in $\cT_1 \setminus \cT_2$ and $u_0$ must be the repeller in this component. If neither happen, then $\cO$ is wandering.

\begin{teo}[{\cite[Lemma 2.17]{GL}}]
Every element of $\mathrm{AAut}(\cT)$ is associated to some revealing tree pair.
\end{teo}

As a consequence we deduce the following corollary, the proof of which uses ideas present in \cite[Section 3.1]{GL} (compare \cite[Lemma 5.5]{hurtadoMiliton} for Higman-Thompson's group $V$). 

\begin{cor} \label{cor:dyn}
If $g \in \mathrm{AAut}(\cT)$ there is a partition $\partial \cT = U_g \sqcup V_g$ into $g$-invariant clopen subsets such that $U_g$ is equal to the open subset of $\xi \in \partial \cT$ that admit a neighborhood $U \subseteq \partial \cT$ such that some positive power of $\restr{g}{U}$ is an isometry onto $U$. Moreover, the following properties are verified.
\begin{itemize}
	\item There is a positive power of $\restr{g}{U_g}$ that is an isometry (for the visual metric of $\partial \cT$).
	\item There are finitely many $g$-periodic points in $V_g$, which we denote $\mathrm{Per}_\mathrm{hyp}(g)$.
	\item There is a partition $\mathrm{Per}_\mathrm{hyp}(g) = \mathrm{Per}_\mathrm{rep}(g) \sqcup \mathrm{Per}_\mathrm{att}(g)$ into \emph{repelling} and \emph{attracting} periodic points, such that for every $\epsilon > 0$, there exists $N \in \N$ so that for all $k \geq N$ we have
	\begin{equation} \label{eq:hyp}
		g^k(V_g \setminus \mathrm{Per}_\mathrm{rep}(g)^\epsilon ) \subseteq \mathrm{Per}_\mathrm{att}(g)^\epsilon \quad \text{ and } \quad  g^{-k}(V_g \setminus \mathrm{Per}_\mathrm{att}(g)^\epsilon ) \subseteq \mathrm{Per}_\mathrm{rep}(g)^\epsilon.
	\end{equation}
\end{itemize}
\end{cor}

\begin{proof}
Let $(\kappa, \cT_1, \cT_2)$ be a revealing pair associated to $g$. Write $\widetilde{U}_g = \bigsqcup \partial \cT_u$ where the union is over all vertices $u \in \cL \cT_1 \cup \cL \cT_2$ that are in periodic chains, and set $\widetilde{V}_g = \partial \cT \setminus \widetilde{U}_g$. The sets $\widetilde{U}_g$ and $\widetilde{V}_g$ are $g$-invariant and clopen, and if $\partial \cT_u \subseteq \widetilde{U}_g$ where $u \in \cL \cT_1 \cup \cL \cT_2$, there exists $n \in \N$ such that $g^n(\partial \cT_u) = \partial \cT_u$ and $\restr{g^n}{\partial \cT_u}$ is an isometry. By taking appropriate powers of $g$ we see that there is a $k \in \N$ such that $\restr{g^k}{\widetilde{U}_g}$ is an isometry.

Now let $\partial \cT_u \subseteq \widetilde{V}_g$ where $u \in \cL \cT_1 \cup \cL \cT_2$ is in an attracting chain $\{u_0,\ldots, u_n\}$. Then $g^n(\partial \cT_{u_0}) \subsetneq \partial \cT_{u_0}$ and $\restr{g^n}{\partial \cT_{u_0}}$ is a contraction, so in particular $g^n(\partial \cT_u) \subsetneq \partial \cT_u$ and $\restr{g^n}{\partial \cT_u}$ is a contraction. Thus there exists a unique $g^n$-fixed point $\xi_u \in \partial \cT_u$, and it also verifies $\bigcap_{k \in \N}g^{kn}(\partial \cT_u) = \{\xi_u\}$. Set 
\[
	\mathrm{Per}_\mathrm{att}(g) = \{\xi_u \mid u \text{ is in an attracting chain}\}.
\] If $\partial \cT_u \subseteq \widetilde{V}_g$ where $u \in \cL\cT_1 \cup \cL \cT_2$ is in a repelling chain instead, the same argument applied to $g^{-1}$ shows that for some $n \in \N$ there exists a unique $g^n$-fixed point $\xi_u \in \partial \cT_u$ such that $\bigcap_{k \in \N} g^{-kn}(\partial \cT_u) = \{\xi_u\}$. Set 
\[
	\mathrm{Per}_\mathrm{rep}(g) = \{\xi_u \mid u \text{ is in a repelling chain}\}.
\]

The sets $\mathrm{Per}_\mathrm{att}(g)$ and $\mathrm{Per}_\mathrm{rep}(g)$ are disjoint and finite. Moreover, their union gives all $g$-periodic points in $\widetilde{V}_g$ since any $\xi \in \partial \cT_u$ where $u \in \cL \cT_1 \cup \cL \cT_2$ is in a wandering chain cannot be a $g$-periodic point: indeed, if the orbit $\{u_0,\ldots, u_n\}$ is wandering, then the connected component of $u_n \in \cL \cT_2 \setminus \cT_1$ in $\cT_2 \setminus \cT_1$ has a root $u_\mathrm{a} \in \cL \cT_1$ which must be the first element of an attracting orbit. Thus 
\[
	d\left(g^{n - k + j}\left(\partial \cT_{u_k}\right), g^j\left(\xi_{u_\mathrm{a}}\right) \right) \xrightarrow{j \to \infty} 0
\] for every $k = 0,\ldots, n$. In the same way, the root $u_\mathrm{r} \in \cL \cT_2$ of the connected component of $u_0 \in \cL \cT_1 \setminus \cT_2$ in $\cT_1 \setminus \cT_2$ is the last element of a repelling orbit, so 
\[
	d\left(g^{- k - j}\left(\partial \cT_{u_k} \right) , g^{-j}\left(\xi_{u_\mathrm{r}}\right) \right) \xrightarrow{j \to \infty} 0
\] for every $k = 0, \ldots, n$. Since $\xi_{u_\mathrm{r}} \neq \xi_{u_\mathrm{a}}$, there are no periodic points in any $\partial \cT_{u_k}$ for $k = 0, \ldots, n$.

To prove that $\widetilde{V}_g,\, \mathrm{Per}_\mathrm{att}(g), \, \mathrm{Per}_\mathrm{rep}(g)$ verify the last item of the corollary, by symmetry it suffices to prove the first statement in \eqref{eq:hyp}. To do this, take a clopen $U \subsetneq \widetilde{V}_g \setminus \mathrm{Per}_\mathrm{rep}(g)$ and  $\epsilon > 0$, and let $u \in \cL \cT_1 \cup \cL \cT_2$. If $u$ is in an attracting or wandering chain, it is clear that there is an $N \in \N$ such that $g^k(U\cap \partial \cT_u) \subseteq \mathrm{Per}_\mathrm{att}(g)^\epsilon$ for all $k \geq N$. If $u$ is in a repelling orbit $\{u_0,\ldots, u_n\}$, then the equality 
\[
	\bigcap_{k \in \N}g^{-kn}(\partial \cT_u) \cap (U\cap \partial \cT_u) = \emptyset
\] shows that there is a $l \in \N$ such that $g^{-kn}(\partial \cT_u) \cap (U\cap \partial \cT_u) = \emptyset$ for all $k \geq l$. The sets $g^r(\partial \cT_u),\, r = 0, \ldots, n-1$ are pairwise disjoint, and we conclude that $\partial \cT_u \cap g^k(U \cap \partial \cT_u) = \emptyset$ for all $k \geq nl$. Upon replacing $U$ by $g^r(U)$ for some sufficiently big $r \in \N$ we may assume that $U \cap \partial \cT_u$ is empty for all $u \in \cL \cT_1 \cup \cL \cT_2$ in a repelling orbit. Hence \eqref{eq:hyp} holds for a sufficiently large $N \in \N$.

Up to now, $\widetilde{U}_g,\, \widetilde{V}_g$ satisfy all the required properties except that $\widetilde{V}_g$ could still contain some points that admit a neighborhood $U \subseteq \partial \cT$ such that some positive power of $\restr{g}{U}$ is an isometry onto $U$. Call the set of such points $\cI \subseteq \widetilde{V}_g$. All elements in $\widetilde{V}_g \setminus (\mathrm{Per}_\mathrm{rep}(g) \sqcup \mathrm{Per}_\mathrm{att}(g))$ admit a neighborhood $U$ such that $g^j(U) \cap U = \emptyset$ for all $j \neq 0$, so $\cI \subseteq \mathrm{Per}_\mathrm{rep}(g) \sqcup \mathrm{Per}_\mathrm{att}(g)$ is finite. Thus setting $U_g = \widetilde{U}_g \sqcup \cI$, $V_g = \widetilde{V}_g \setminus  \cI$ ensures that all properties in the statement of the corollary are verified.
\end{proof}
In contrast with the trees $\cT_{d,k}$, the boundary of an arbitrary locally finite tree $\cT$ may contain isolated points and the three items in the previous corollary do not specify $U_g, V_g$ uniquely. The definition of $U_g$ in the statement of the corollary gives a canonical definition in any case.

\subsection*{Vietoris topology}
Given a topological space $X$, we denote by $2^X$ the space of closed subsets of $X$ equipped with the \emph{Vietoris topology}, defined as the topology with subbasis
\[
	\{K \in 2^X \mid K \cap U \neq \emptyset,\, K \cap V = \emptyset\}
\] where $U,V$ range over all open subsets of $X$. If $X$ is compact metrizable then $2^X$ is also compact metrizable \cite[Section 4.5]{engelking}, and in this case any action by homeomorphisms of a countable group on $X$ induces naturally an action by homeomorphisms on $2^X$ \cite[Section 3.12]{engelking}.

\section{Proofs}
We first gather some useful lemmas.

\begin{lema}[{Neumann's lemma, \cite[Lemma 4.1]{neumann}}] \label{lema:neumann}
Let $G$ be a group acting on a set $X$ and assume that the action has no finite orbits.	Then for every pair of finite subsets $A,B$ of $X$ there exists an element $g \in G$ such that $g(A) \cap B$ is empty.
\end{lema}

\begin{lema}[{\cite[Proposition 1.16]{malicetMiliton}}] \label{lema:libre}
Let $G$ be a group acting by homeomorphisms on a compact metric space $(X,d)$. Assume that
\begin{enumerate}
	\item \label{item:pingPong1} there exists a positive integer $p$ such that for any $\varepsilon > 0$ there exist nonempty finite sets $A,B \subset X$ of cardinality at most $p$ with $g(X \setminus A^\epsilon) \subseteq B^\epsilon$ for some $g \in G$, and
	\item \label{item:pingPong2} there are no finite $G$-orbits.
\end{enumerate}
Then the action of $G$ on $X$ has a ping-pong pair.
\end{lema}

\begin{proof}
Notice that condition \eqref{item:pingPong1} is equivalent to the existence of finite sets $A,B$ of $X$ that work for \emph{any} $\epsilon >0$: a pair $(A,B)$ of nonempty subsets of $X$ of cardinality at most $p$ is called a \textit{contraction pair} if for every neighborhood $U, V$ of $A,B$ respectively there exists $g \in G$ such that $g(X\setminus U) \subset V$. Then condition \eqref{item:pingPong1} and the compactness of $X$ imply that there exists a contraction pair.

Moreover, if $(A,B)$ is a contraction pair and $u,v \in G$ then $(u(A), v(B))$ is a contraction pair too. Thus, by \eqref{item:pingPong2} and Lemma \ref{lema:neumann}, our contraction pair $(A,B)$ can be taken such that $A$ and $B$ are disjoint.

Using \eqref{item:pingPong2} and Lemma \ref{lema:neumann} again we deduce that there exist two contraction pairs $(A_1, B_1), (A_2, B_2)$ where the $A_1, A_2, B_1, B_2$ are pairwise disjoint. If $U_i, V_i, i = 1,2$ are pairwise disjoint neighborhoods of $A_i, B_i, i = 1,2$ respectively, we can find $g_1, g_2 \in G$ such that $g_i(X \setminus U_i) \subseteq V_i,\, i = 1,2$. These constitute a ping-pong pair.
\end{proof}

\begin{lema} \label{lema:rec}
Let $G$ be a compact topological group and $g \in G$. For every neighborhood $U$ of the identity there exists a strictly increasing sequence $\{n_j\}_{j \in \N} \subseteq \N$ such that $g^{n_j} \in U$ for all $j \in \N$.
\end{lema}

\begin{proof}
Let $\mu$ be the normalized Haar measure on $G$, so that $\mu$ is a probability measure of complete support on $G$ and left multiplication by $g$ preserves $\mu$. Let $V \subseteq U$ be an open neighborhood of the identity such that $V \cdot V^{-1} \subseteq U$. Since $\mu(V) > 0$, Poincaré's recurrence theorem implies that there is a sequence $\{n_j\}_{j \in \N} \subseteq \N$ such that $\mu(V \cap g^{n_j}V) > 0$ for all $j \in \N$. In particular $g^{n_j} \in V\cdot V^{-1} \subseteq U$ for all $j \in \N$.
\end{proof}

The core of the proof of Theorem \ref{teo:neretin} is the following statement, which uses ideas from \cite[Proposition 5.2]{leBoudecMatteBon} and implies the extreme contraction properties required by Lemma \ref{lema:libre}. If $g \in \mathrm{AAut}(\cT)$ we use the same notation as Corollary \ref{cor:dyn}, so we write $U_g \subseteq \partial \cT$ for the stable set of $g$ and $\mathrm{Per}_\mathrm{hyp}(g),\, \mathrm{Per}_\mathrm{rep}(g)$ for the hyperbolic and repelling points of $g$ respectively.

Recall that we denote by $2^{\partial \cT}$ the compact metrizable space of all closed subsets of $\partial \cT$ equipped with the Vietoris topology.

\begin{prop} \label{prop:proximal}
Let $H \leq \mathrm{AAut}(\cT)$ and assume that $\bigcap_{h \in H}U_h$ is empty. Then there exists a finite set $B \subset \partial \cT$ such that for every $\epsilon > 0$ there is a $h \in H$ with $h(\partial \cT \setminus B^\epsilon) \subseteq B^\epsilon$.
\end{prop}

\begin{proof}
By compactness there exist elements $h_1, \ldots, h_k \in H$ such that $\bigcap_{1\leq j \leq k}U_{h_j}$ is empty. By taking powers of the $h_j$ we can assume that every $\restr{h_j}{U_j}$ is an isometry. Set $B = \bigcup_{1\leq j \leq k}\mathrm{Per}_{\mathrm{hyp}}(h_j)$ and $C_1 = \partial \cT \setminus B^\epsilon$.

Now $h_1$ restricted to $U_{h_1}$ is an isometry. Since isometry groups of compact spaces are compact for the compact-open topology, Lemma \ref{lema:rec} and a diagonal argument shows that there exists a strictly increasing sequence $\{n_j\}_{j \in \N}\subseteq \N$ such that
\[
	\sup_{x \in U_{h_1}}d\left(h_1^{n_j}(x),x\right) \leq \frac{1}{j}
\] for all $j \in \N$. 

For any closed $D \subseteq \partial \cT$ denote by $\overline{\mathrm{Orb}_H(D)}$ the closure of the $H$-orbit of $D$ in $2^{\partial \cT}$. By taking a limit point of $\{h_1^{n_j}(C_1)\}_{j \in \N}$ in $2^{\partial \cT}$ we obtain a closed $C_2 \in \overline{\mathrm{Orb}_H(C_1)}$ with
\[
	C_2 \subseteq (C_1 \cap U_{h_1}) \sqcup \mathrm{Per}_{\mathrm{hyp}}(h_1).
\]

Lemma \ref{lema:rec} applied to $h_2$ gives again the existence of a sequence $\{n_j'\}_{j \in \N}\subseteq \N$ such that
\[
	\sup_{x \in U_{h_2}}d\left(h_2^{n_j'}(x), x\right) \leq \frac{1}{j}
\] for all $j \in \N$. The finite set $\mathrm{Per}_\mathrm{rep}(h_2)$ can only intersect $C_2$ in $\mathrm{Per}_{\mathrm{hyp}}(h_1)$, so by taking again a limit point of $\{h_2^{n_j'}(C_2)\}_{j \in \N}$ we obtain a closed $C_3 \in \overline{\mathrm{Orb}_H(C_2)} \subseteq \overline{\mathrm{Orb}_H(C_1)}$ with
\[
	C_3 \subseteq (C_2 \cap U_{h_2}) \sqcup \mathrm{Per}_{\mathrm{hyp}}(h_2) \subseteq (C_1 \cap U_{h_1}\cap U_{h_2}) \sqcup \left( \mathrm{Per}_{\mathrm{hyp}}(h_1) \cup \mathrm{Per}_\mathrm{hyp}(h_2)\right).
\] Notice that, again, the finite set $\mathrm{Per}_\mathrm{rep}(h_{3})$ can only intersect $C_3$ in $\mathrm{Per}_\mathrm{hyp}(h_1) \cup \mathrm{Per}_\mathrm{hyp}(h_2)$.

By iterating this argument we produce, for every $j = 2 ,\ldots, k+1$, a closed subset $C_j \in \overline{\mathrm{Orb}_H(C_{j-1})} \subseteq \overline{\mathrm{Orb}_H(C_1)}$ such that
\[
	C_j \subseteq (C_{j-1} \cap U_{h_{j-1}}) \sqcup \mathrm{Per}_\mathrm{hyp}(h_{j-1}) \subseteq (C_1 \cap U_{h_1} \cap \cdots \cap U_{h_{j-1}}) \sqcup \left( \bigcup_{1 \leq i \leq j-1}\mathrm{Per}_{\mathrm{hyp}}(h_i) \right).
\] In particular $C_{k+1} \subseteq \bigcup_{1 \leq i \leq k} \mathrm{Per}_{\mathrm{hyp}}(h_i) = B$, and we are done.
\end{proof}

\begin{proof}[\textbf{Proof of Theorem \ref{teo:neretin}}]
Take a subgroup $H \leq \mathrm{AAut}(\cT)$ and suppose that $H$ does not preserve a probability measure on $\partial \cT$. If every finitely generated subgroup of $H$ preserves a probability measure on $\partial \cT$, the compactness of $\mathrm{Prob}(\partial \cT)$ equipped with the weak-$\ast$ topology implies that $H$ preserves a probability measure on $\partial \cT$. We may assume then that $H$ is finitely generated.

Suppose that the set $U_H = \bigcap_{h \in H}U_h$ is nonempty. Since every $U_h$ is dynamically defined (see the definition of $U_h$ in Corollary \ref{cor:dyn}), we have that $U_H$ is $H$-invariant. The closed set $U_H$ may be written as $\partial \cT'$ for some rooted subtree $\cT' \subseteq \cT$, and we equip $U_H$ with the restriction of the visual metric of $\partial \cT$, which coincides with the visual metric of $\partial \cT'$. Every element of $H$ has a proper power that acts as an isometry on $\partial \cT'$, so the action of $H$ on $\partial \cT'$ is by elliptic almost automorphisms of $\cT'$. Since $H$ is finitely generated, Proposition \ref{prop:elliptic} shows that the image of $H$ in $\mathrm{AAut}(\cT')$ is relatively compact. Thus $H$ preserves a measure on $\partial \cT' \subseteq \partial \cT$, a contradiction. We conclude that $U_H$ is empty, and in this case Proposition \ref{prop:proximal} and Lemma \ref{lema:libre} imply that there exists a ping-pong pair for the action of $H$.
\end{proof}

\begin{rmk}
The proof of Theorem \ref{teo:neretin} shows a slightly stronger statement, namely that for any finitely generated $H \leq \mathrm{AAut}(\cT)$, either $H$ contains a ping-pong pair or there exists a nonempty closed set $C \subseteq \partial \cT$ such that the action of $H$ on $C$ is equicontinuous. 
\end{rmk}

Recall that $V_\cT$ denotes the Higman-Thompson group associated to $\cT$.

\begin{proof}[\textbf{Proof of Theorem \ref{teo:V}}]
Fix a linear order on $\partial \cT$ as in the definition of $V_\cT$. Take $H \leq V_\cT$ a finitely generated subgroup and suppose that $H$ acts without finite orbits on $\partial \cT$. If $U_H = \bigcap_{h \in H} U_h$ is nonempty, consider again a rooted subtree $\cT' \subseteq \cT$ such that $U_H = \partial \cT'$, so $H$ acts on $U_H$ by elliptic Higman-Thompson elements of $\mathrm{AAut}(\cT')$ for the induced order on $\partial \cT'$. Corollary \ref{cor:torsion} shows that the image of $H$ in $\mathrm{AAut}(\cT')$ is finite, and thus there exists a finite $H$-orbit in $U_H$. This is a contradiction, hence $U_H$ is empty. Again Proposition \ref{prop:proximal} and Lemma \ref{lema:libre} show that there exists a ping-pong pair for the action of $H$.
\end{proof}

\section{Open questions}
We finish with some hopefully tractable open questions.
\begin{qn}
R. Grigorchuk, V. Nekrashevych and V. Sushchansky introduce in \cite{GNS} the group $\cR$ of homeomorphisms of Cantor space defined by asynchronous transducers, which is known \cite{BB} to contain the higher-dimensional Brin-Thompson groups from \cite{2V}. On the other hand, the automorphism group of a Higman-Thompson group $V_{d,k}$ coincides with the subgroup of bi-synchronizing transducers inside $\cR$ \cite{BCMNO}. Do any of these groups satisfy the dynamical Tits alternative?
\end{qn}

\begin{qn}
Let $G$ be a group of homeomorphisms of a compact topological space $X$ such that its groupoid of germs is hyperbolic in the sense of V. Nekrashevych, see \cite{hypGpds}. Does the action of $G$ on $X$ satisfy the dynamical Tits alternative? Theorem \ref{teo:neretin} shows that this is true for the groups $V_G$, whose groupoid of germs is hyperbolic whenever $G$ is a contracting self-similar group.
\end{qn}

\newcommand{\etalchar}[1]{$^{#1}$}
\providecommand{\bysame}{\leavevmode\hbox to3em{\hrulefill}\thinspace}

\providecommand{\href}[2]{#2}
\ProvideTextCommandDefault{\cprime}{\tprime}


\begin{thebibliography}{BCGM12}

\bibitem[AS22]{aounSert}
Richard Aoun and Cagri Sert, \href{http://dx.doi.org/10.1007/s00440-022-01116-1}{\emph{Random walks on hyperbolic spaces: concentration inequalities and probabilistic {Tits} alternative}}, Probability Theory and Related Fields \textbf{184} (2022), no.~1-2, pp.~323--365.

\bibitem[BB17]{BB}
James Belk and Collin Bleak, \href{https://doi.org/10.1090/tran/6963}{\emph{Some undecidability results for asynchronous transducers and the {B}rin-{T}hompson group {$2V$}}}, Transactions of the American Mathematical Society \textbf{369} (2017), no.~5, pp.~3157--3172.

\bibitem[BCGM12]{BCGM}
Uri Bader, Pierre-Emmanuel Caprace, Tsachik Gelander, and Shahar Mozes, \href{https://doi.org/10.1112/blms/bdr061}{\emph{Simple groups without lattices}}, Bulletin of the London Mathematical Society \textbf{44} (2012), no.~1, pp.~55--67.

\bibitem[BCM{\etalchar{+}}19]{BCMNO}
Collin Bleak, Peter Cameron, Yonah Maissel, Andrés Navas, and Feyishayo Olukoya, \href{https://arxiv.org/abs/1605.09302}{\emph{The further chameleon groups of {Richard Thompson and Graham Higman: Automorphisms via dynamics for the Higman groups $G_{n,r}$}}} (2019 preprint), arXiv:1605.09302.

\bibitem[Bri04]{2V}
Matthew~G. Brin, \href{https://doi.org/10.1007/s10711-004-8122-9}{\emph{Higher dimensional {T}hompson groups}}, Geometriae Dedicata \textbf{108} (2004), pp.~163--192.

\bibitem[CCMT15]{CCMT}
Pierre-Emmanuel Caprace, Yves Cornulier, Nicolas Monod, and Romain Tessera, \href{http://dx.doi.org/10.4171/JEMS/575}{\emph{Amenable hyperbolic groups}}, Journal of the European Mathematical Society \textbf{17} (2015), no.~11, pp.~2903--2947.

\bibitem[CFP96]{CFP}
James~W. Cannon, William~J. Floyd, and Walter~R. Parry, \emph{Introductory notes on {Richard} {Thompson}'s groups}, L'Enseignement Math{\'e}matique. 2e S{\'e}rie \textbf{42} (1996), no.~3-4, pp.~215--256.

\bibitem[dlH00]{deLaHarpeTopics}
Pierre de~la Harpe, \emph{Topics in geometric group theory}, Chicago Lectures in Mathematics, University of Chicago Press, Chicago, IL, 2000.

\bibitem[DM18]{DM}
Bruno Duchesne and Nicolas Monod, \href{http://dx.doi.org/10.5802/aif.3209}{\emph{Group actions on dendrites and curves}}, Annales de l'Institut Fourier \textbf{68} (2018), no.~5, pp.~2277--2309.

\bibitem[Eng89]{engelking}
Ryszard Engelking, \emph{General topology}, second ed., Sigma Series in Pure Mathematics, vol.~6, Heldermann Verlag, Berlin, 1989, Translated from the Polish by the author.

\bibitem[GL18]{garncarekLazarovich}
Lukasz Garncarek and Nir Lazarovich, \emph{The {N}eretin groups}, New directions in locally compact groups, London Math. Soc. Lecture Note Ser., vol. 447, Cambridge Univ. Press, Cambridge, 2018, pp.~131--144.

\bibitem[GL21]{GL}
Gil Goffer and Waltraud Lederle, \href{http://dx.doi.org/10.1142/S0218196721500557}{\emph{Conjugacy and dynamics in almost automorphism groups of trees}}, International Journal of Algebra and Computation \textbf{31} (2021), no.~8, pp.~1497--1545.

\bibitem[GNS00]{GNS}
Rostislav~I. Grigorchuk, Volodymyr~V. Nekrashevych, and Vitaly~I. Sushchansky, \emph{Automata, dynamical systems, and groups}, Dynamical systems, automata, and infinite groups, Moscow: MAIK Nauka/Interperiodica Publishing, 2000, pp.~128--203.

\bibitem[Gri80]{grigorchuk}
Rostislav~I. Grigorchuk, \emph{On {B}urnside's problem on periodic groups}, Akademiya Nauk SSSR. Funktsional ny\u{\i} Analiz i ego Prilozheniya \textbf{14} (1980), no.~1, pp.~53--54.

\bibitem[HM19]{hurtadoMiliton}
Sebasti{á}n Hurtado and Emmanuel Militon, \href{http://dx.doi.org/10.1090/tran/7476}{\emph{Distortion and {Tits} alternative in smooth mapping class groups}}, Transactions of the American Mathematical Society \textbf{371} (2019), no.~12, pp.~8587--8623.

\bibitem[Kap99]{kapoudjian}
Christophe Kapoudjian, \href{http://dx.doi.org/10.5802/aif.1715}{\emph{Simplicity of {Neretin}'s group of spheromorphisms}}, Annales de l'Institut Fourier \textbf{49} (1999), no.~4, pp.~1225--1240.

\bibitem[LB17]{leBoudecCompact}
Adrien Le~Boudec, \href{http://dx.doi.org/10.5802/aif.3084}{\emph{Compact presentability of tree almost automorphism groups}}, Annales de l'Institut Fourier \textbf{67} (2017), no.~1, pp.~329--365.

\bibitem[LBMB22]{leBoudecMatteBon}
Adrien Le~Boudec and Nicol\'{a}s Matte~Bon, \href{https://doi.org/10.5802/ahl.128}{\emph{Confined subgroups and high transitivity}}, Annales Henri Lebesgue \textbf{5} (2022), pp.~491--522.

\bibitem[LBW19]{leBoudecWesolek}
Adrien Le~Boudec and Phillip Wesolek, \href{http://dx.doi.org/10.4171/GGD/477}{\emph{Commensurated subgroups in tree almost automorphism groups}}, Groups, Geometry, and Dynamics \textbf{13} (2019), no.~1, pp.~1--30.

\bibitem[Led20]{lederleV}
Waltraud Lederle, \href{https://doi.org/10.1007/s00208-020-02063-9}{\emph{Topological full groups and t.d.l.c. completions of {T}hompson's {$V$}}}, Mathematische Annalen \textbf{378} (2020), no.~3-4, pp.~1415--1434.

\bibitem[Mar00]{margulisCircle}
Gregory Margulis, \href{http://dx.doi.org/10.1016/S0764-4442(00)01694-3}{\emph{Free subgroups of the homeomorphism group of the circle}}, Comptes Rendus de l'Acad{\'e}mie des Sciences. S{\'e}rie I. Math{\'e}matique \textbf{331} (2000), no.~9, pp.~669--674.

\bibitem[MM23]{malicetMiliton}
Dominique Malicet and Emmanuel Militon, \href{http://arxiv.org/abs/2304.08070}{\emph{Random actions of homeomorphisms of {Cantor} sets embedded in a line and {Tits} alternative}} (2023 preprint), arXiv:2304.08070.

\bibitem[Nek15]{hypGpds}
Volodymyr~V. Nekrashevych, \href{https://doi.org/10.1090/memo/1122}{\emph{Hyperbolic groupoids and duality}}, Memoirs of the American Mathematical Society \textbf{237} (2015), no.~1122, pp.~v+105.

\bibitem[Nek18]{nekrashevych}
\bysame, \emph{Finitely presented groups associated with expanding maps}, Geometric and cohomological group theory, London Math. Soc. Lecture Note Ser., vol. 444, Cambridge Univ. Press, Cambridge, 2018, pp.~115--171.

\bibitem[Neu54]{neumann}
Bernhard~H. Neumann, \href{https://doi.org/10.1112/jlms/s1-29.2.236}{\emph{Groups covered by permutable subsets}}, Journal of the London Mathematical Society \textbf{29} (1954), pp.~236--248.

\bibitem[SD10]{salazarDiaz}
Olga~Patricia Salazar-D\'{\i}az, \href{https://doi.org/10.1142/S0218196710005534}{\emph{Thompson's group {$V$} from a dynamical viewpoint}}, International Journal of Algebra and Computation \textbf{20} (2010), no.~1, pp.~39--70.

\bibitem[Tit70]{titsArbre}
Jacques Tits, \emph{Sur le groupe des automorphismes d'un arbre}, Essays on Topology and Related Topics, {M{\'e}moires} d{\'e}di{\'e}s {\`a} {Georges} de {Rham}, 188-211, 1970.

\bibitem[Tit72]{titsFree}
\bysame, \href{https://doi.org/10.1016/0021-8693(72)90058-0}{\emph{Free subgroups in linear groups}}, Journal of Algebra \textbf{20} (1972), pp.~250--270.

\bibitem[Zhe19]{zheng}
Tianyi Zheng, \href{http://arxiv.org/abs/1905.07605}{\emph{Neretin groups admit no non-trivial invariant random subgroups}} (2019 preprint), arXiv:1905.07605.

\end{thebibliography}
\end{document}